\documentclass[10pt, reqno]{amsart}

\usepackage[centertags]{amsmath}
\usepackage{amsfonts}
\usepackage{amssymb}
\usepackage{amsthm}
\usepackage{newlfont}
\usepackage{fancyhdr}

\def\BibTeX{{\rm B\kern-.05em{\sc i\kern-.025em b}\kern-.08em
    T\kern-.1667em\lower.7ex\hbox{E}\kern-.125emX}}

\newcommand{\EE}{\mathbb{E}}

\newcommand{\E}{\mathrm{E}}
\newcommand{\PP}{\mathrm{P}}
    \newcommand{\dto}{\xrightarrow{d}}
    \newcommand{\wto}{\xrightarrow{w}}
    \newcommand{\vto}{\xrightarrow{v}}
    \newcommand{\fidi}{\xrightarrow{\text{fidi}}}

\newcommand{\eqd}{\stackrel{d}{=}}
 \newcommand{\floor}[1]{\lfloor#1\rfloor}
 \newcommand{\rmd}{\mathrm{d}}
\theoremstyle{plain}
\newtheorem{thm}{Theorem}[section]

\newtheorem{lem}[thm]{Lemma}

\theoremstyle{definition}
\newtheorem{rem}[thm]{Remark}
\newtheorem{ex}[thm]{Example}
\newtheorem{cond}[thm]{Condition}
\numberwithin{equation}{section}

\begin{document}

\title[On joint weak convergence of partial sum and maxima processes] 
{On joint weak convergence of partial sum and maxima processes}

%
\author{Danijel Krizmani\'{c}}

\address{Danijel Krizmani\'{c}\\ Department of Mathematics\\
        University of Rijeka\\
        Radmile Matej\v{c}i\'{c} 2, 51000 Rijeka\\
        Croatia}
\email{dkrizmanic@math.uniri.hr}




\begin{abstract}
For a strictly stationary sequence of random variables we derive functional convergence of the joint partial sum and partial maxima process under joint regular variation with index $\alpha \in (0,2)$ and weak dependence conditions. The limiting process consists of an $\alpha$--stable L\'{e}vy process and an extremal process. We also describe the dependence between these two components of the limit. The convergence takes place in the space of $\mathbb{R}^{2}$--valued c\`{a}dl\`{a}g functions on
$[0,1]$, with the Skorohod weak $M_{1}$ topology. We further show that this topology in general can not be replaced by the stronger (standard) $M_{1}$ topology.
\end{abstract}

\maketitle

\section{Introduction}

Consider a strictly stationary sequence of random variables $(X_{n})$ and denote by $S_{n}= X_{1} + \ldots + X_{n}$ and $M_{n}= \max\{X_{i}
 : i=1, \ldots, n \}$, $n \geq 1$, its accompanying sequences of partial sums and maxima, respectively. It is well known that if the $X_{n}$ are i.i.d.~and regularly varying with index $\alpha \in (0,2)$, then
 $$ \frac{S_{n}-b_{n}'}{a_{n}'} \dto S,$$
 for some $a_{n}'>0$ and $b_{n}' \in \mathbb{R}$ and some $\alpha$--stable random variable $S$, and
 $$\frac{M_{n}}{a_{n}''} \dto Y,$$
 for some $a_{n}''>0$ and some random variable $Y$ with an extreme-value distribution, see for example Gnedenko and Kolmogorov~\cite{GnKo54} and Resnick~\cite{Re87}. The weak convergence of partial maxima holds also for $\alpha \geq 2$.
 The joint weak limiting behavior of $(S_{n}, M_{n})$ with appropriate centering and scaling was investigated by Chow and Teugels~\cite{ChTe78}. They also obtained a functional limit theorem for a suitably normalized joint partial sum and partial maxima process. See also Anderson and Turkman~\cite{AnTu91} and Resnick~\cite{Re86} for related results.

In this paper, under the properties of weak dependence and joint regular variation with index $\alpha \in (0,2)$ for the sequence $(X_{n})$, we investigate functional convergence of the joint partial sum and partial maxima process $L_{n} = (V_{n}, W_{n})$ in the space $D([0,1], \mathbb{R}^{2})$, where
$$V_{n}(t) =  \frac{S_{\floor{nt}}}{a_{n}} - \floor{nt}b_{n},  \quad    W_{n}(t) = \frac{ M_{\lfloor nt \rfloor}}{a_{n}},  \qquad t \in [0,1],$$
with $(a_{n})$ being a sequence of positive real numbers such that
\begin{equation}\label{e:niz}
 n \PP ( |X_{1}| > a_{n}) \to 1,
\end{equation}
as $n \to \infty$ and
\begin{equation}\label{e:centb}
 b_{n} = \left\{ \begin{array}{cc}
                                   0, & \quad \alpha \in (0,1),\\[0.5em]
                                   \mathrm{E} \Big( \frac{X_{1}}{a_{n}} 1_{ \big\{\frac{|X_{1}|}{a_{n}} \leq 1 \big\} } \Big), & \quad \alpha \in [1,2).
                                 \end{array}\right.
\end{equation}
Here, $\floor{x}$ represents the greatest integer not larger than $x$ and $D([0, 1], \mathbb{R}^{2})$ is the space of $\mathbb{R}^{2}$--valued c\`adl\`ag functions on $[0, 1]$. The functional convergence of $L_{n}$ in the special case when $(X_{n})$ is a linear process from a regularly varying distribution with index $\alpha \in (0,2)$ was studied recently in Krizmani\'{c}~\cite{Kr18-0},~\cite{Kr18}.

The main result of our article shows that for a strictly stationary, regularly varying sequence of dependent random variables $(X_{n})$ with index $\alpha \in (0,2)$, for which clusters of high-treshold excesses can be broken down into asymptotically independent blocks, the stochastic process $L_{n}$ converge in the space $D([0,1], \mathbb{R}^{2})$ endowed with the Skorohod weak $M_{1}$ topology under the condition that all extremes within each cluster of big values have the same sign. This topology is weaker than the more commonly used Skorohod $J_{1}$ topology, the latter being appropriate when there is no clustering of extremes (which for example occurs in the i.i.d.~case).

The paper is organized as follows. In Section~\ref{s:one} we introduce the essential ingredients about regular variation, weak dependence and Skorohod topologies. In Section~\ref{s:three} we state and prove our main result using a new limit theorem derived recently by Basrak and Tafro~\cite{BaTa16} for the time-space point processes $N_{n} = \sum_{i=1}^{n}\delta_{(i/n, X_{i}/a_{n})}$.
Finally, in Section~\ref{s:four} we illustrate by an example that
the weak $M_{1}$ convergence in our main theorem, in general, can not be replaced by the standard $M_{1}$ convergence.

\section{Preliminaries}\label{s:one}

\subsection{Regular variation}

Let $\EE^{d}=[-\infty, \infty]^{d} \setminus \{ 0 \}$. We equip
$\EE^{d}$ with the topology in which a set $B \subset \EE^{d}$
has compact closure if and only if it is bounded away from zero,
that is, if there exists $u > 0$ such that $B \subset \EE^{d}_u = \{ x
\in \EE^{d} : \|x\| >u \}$. Here $\| \cdot \|$ denotes the max-norm on $\mathbb{R}^{d}$, i.e.\
$\displaystyle \| x \|=\max \{ |x_{i}| : i=1, \ldots , d\}$ where
$x=(x_{1}, \ldots, x_{d}) \in \mathbb{R}^{d}$. Denote by $C_{K}^{+}(\EE^{d})$ the class of all
nonnegative, continuous functions on $\EE^{d}$ with compact support.

We say that a strictly stationary process $(X_{n})_{n \in \mathbb{Z}}$ is \emph{(jointly) regularly varying} with index
$\alpha \in (0,\infty)$ if for any nonnegative integer $k$ the
$kd$-dimensional random vector $X = (X_{1}, \ldots , X_{k})$ is
multivariate regularly varying with index $\alpha$, i.e.\ there
exists a random vector $\Theta$ on the unit sphere
$\mathbb{S}^{kd-1} = \{ x \in \mathbb{R}^{kd} : \|x\|=1 \}$ such
that for every $u \in (0,\infty)$ and as $x \to \infty$,
 \begin{equation}\label{e:regvar1}
   \frac{\PP(\|X\| > ux,\,X / \| X \| \in \cdot \, )}{\PP(\| X \| >x)}
    \wto u^{-\alpha} \PP( \Theta \in \cdot \,),
 \end{equation}
the arrow ''$\wto$'' denoting weak convergence of finite measures.
Regular variation can be expressed in terms of vague convergence of
measures on $\EE$ as follows: for $a_n$ as in (\ref{e:niz}),
\begin{equation}
  \label{e:onedimregvar}
  n \PP( a_n^{-1} X_i \in \cdot \, ) \vto \mu( \, \cdot \,),
\end{equation}
where the limit $\mu$ is a nonzero Radon measure on $\EE$ given by
\begin{equation}\label{e:mu}
  \mu(\rmd x) = \bigl( p \, 1_{(0, \infty)}(x) + q \, 1_{(-\infty, 0)}(x) \bigr) \, \alpha |x|^{-\alpha-1}\,\rmd x,
\end{equation}
for some $p \in [0,1]$, with $q=1-p$.

Theorem 2.1 in Basrak and Segers~\cite{BaSe} provides a convenient
characterization of joint regular variation:~it is necessary and
sufficient that there exists a process $(Y_n)_{n \in \mathbb{Z}}$
with $\PP(|Y_0| > y) = y^{-\alpha}$ for $y \geq 1$ such that, as $x
\to \infty$,
\begin{equation}\label{e:tailprocess}
  \bigl( (x^{-1}\ X_n)_{n \in \mathbb{Z}} \, \big| \, | X_0| > x \bigr)
  \fidi (Y_n)_{n \in \mathbb{Z}},
\end{equation}
where "$\fidi$" denotes convergence of finite-dimensional
distributions. The process $(Y_{n})$ is called
the \emph{tail process} of $(X_{n})$.

\subsection{Point processes and dependence conditions}\label{s:pp}

Let $(X_{n})$ be a strictly stationary sequence of random variables and assume it is jointly regularly varying with index $\alpha >0$. Let $(Y_{n})$ be
the tail process of $(X_{n})$. In order to obtain weak convergence of the process $L_{n}$ we will use the so-called complete convergence result for the corresponding point process of jumps obtained recently by Basrak and Tafro~\cite{BaTa16}, and then by the continuous mapping theorem and some properties of Skorohod topologies we will transfer this convergence result to the joint partial sum and maxima process.

Let
\begin{equation*}\label{E:ppspacetime}
 N_{n} = \sum_{i=1}^{n} \delta_{(i / n,\,X_{i} / a_{n})} \qquad \textrm{for all} \ n\in \mathbb{N},
\end{equation*}
with $a_{n}$ as in (\ref{e:niz}). The point process convergence for the sequence $(N_{n})$ was already established by Basrak et al.~\cite{BKS} on the space $[0,1] \times \EE_{u}$ for any threshold $u>0$, with the limit depending on that threshold. Recently Basrak and Tafro~\cite{BaTa16} obtained a new convergence result for $N_{n}$ without the restriction to various domains (i.e. their convergence result holds on the space $[0,1] \times \EE$).

 The appropriate weak dependence conditions for this convergence result are given below. With them we will be able to control the dependence in the sequence $(X_{n})$.

\begin{cond}\label{c:mixcond1}
There exists a sequence of positive integers $(r_{n})$ such that $r_{n} \to \infty $ and $r_{n} / n \to 0$ as $n \to \infty$ and such that for every $f \in C_{K}^{+}([0,1] \times \mathbb{E})$, denoting $k_{n} = \lfloor n / r_{n} \rfloor$, as $n \to \infty$,
\begin{equation}\label{e:mixcon}
 \E \biggl[ \exp \biggl\{ - \sum_{i=1}^{n} f \biggl(\frac{i}{n}, \frac{X_{i}}{a_{n}}
 \biggr) \biggr\} \biggr]
 - \prod_{k=1}^{k_{n}} \E \biggl[ \exp \biggl\{ - \sum_{i=1}^{r_{n}} f \biggl(\frac{kr_{n}}{n}, \frac{X_{i}}{a_{n}} \biggr) \biggr\} \biggr] \to 0.
\end{equation}
\end{cond}

\begin{cond}\label{c:mixcond2}
There exists a sequence of positive integers $(r_{n})$ such that $r_{n} \to \infty $ and $r_{n} / n \to 0$ as $n \to \infty$ and such that for every $u > 0$,
\begin{equation}
\label{e:anticluster}
  \lim_{m \to \infty} \limsup_{n \to \infty}
  \PP \biggl( \max_{m \leq |i| \leq r_{n}} | X_{i} | > ua_{n}\,\bigg|\,| X_{0}|>ua_{n} \biggr) = 0.
\end{equation}
\end{cond}
Condition~\ref{c:mixcond1} is implied by the strong mixing property (see Krizmani\'{c}~\cite{Kr16}).
For a discussion of Conditions~\ref{c:mixcond1} and~\ref{c:mixcond2} we refer to Bartkiewicz et al.~\cite{BaJaMiWi11} and Basrak et al.~\cite{BKS}.
There are many time series satisfying these conditions, including moving averages, stochastic volatility and GARCH models (see for example Basrak et al.~\cite{BKS}, section 4).
By Proposition~4.2 in Basrak and Segers~\cite{BaSe},
under Condition~\ref{c:mixcond2} the following
holds
\begin{eqnarray}\label{E:theta:spectral}
   \theta := \PP ({\textstyle\sup_{i\ge 1}} \| Y_{i}\| \le 1) = \PP ({\textstyle\sup_{i\le -1}} \| Y_{i}\| \le 1)>0,
\end{eqnarray}
and $\theta$ is the extremal index of the univariate sequence $(| X_{n} |)$.
For a detailed discussion on joint regular variation and dependence Conditions~\ref{c:mixcond1} and \ref{c:mixcond2} we refer to Basrak et al.~\cite{BKS}, Section 3.4.

Under joint regular variation and Conditions~\ref{c:mixcond1} and \ref{c:mixcond2}, by Theorem 3.1 in Basrak and Tafro~\cite{BaTa16}, as $n \to \infty$,
\begin{equation}\label{e:BaTa}
N_{n} \dto N = \sum_{i}\sum_{j}\delta_{(T_{i}, P_{i}\eta_{ij})}
\end{equation}
in $[0,1] \times \EE$, where $\sum_{i=1}^{\infty}\delta_{(T_{i}, P_{i})}$ is a Poisson process on $[0,1] \times (0,\infty)$
with intensity measure $Leb \times \nu$ where $\nu(\rmd x) = \theta \alpha
x^{-\alpha-1}1_{(0,\infty)}(x)\,\rmd x$, and $(\sum_{j= 1}^{\infty}\delta_{\eta_{ij}})_{i}$ is an i.i.d.~sequence of point processes in $\EE$ independent of $\sum_{i}\delta_{(T_{i}, P_{i})}$ and with common distribution equal to the distribution of $\sum_{j}\delta_{Z_{j}/L_{Z}}$, where $L_{Z}= \sup_{j \in \mathbb{Z}}|Z_{j}|$ and $\sum_{j}\delta_{Z_{j}}$ is distributed as $( \sum_{j \in \mathbb{Z}} \delta_{Y_j} \,|\, \sup_{i \le -1} | Y_i| \le 1).$

Denote by $\sum_{j}\delta_{\eta_{j}}$ a point process with the distribution equal to the distribution of $\sum_{j}\delta_{\eta_{1j}}$.
For $ \alpha \leq 1$ it holds that
\begin{equation}\label{e:MW1}
 \mathrm{E} \Big( \sum_{j}|\eta_{j}| \Big)^{\alpha} < \infty
\end{equation}
(see Davis and Hsing~\cite{DaHs95}), but it may fail for $\alpha > 1$ (see Mikosch and Wintenberger~\cite{MiWi14}). In the latter case we will have to assume (\ref{e:MW1}).

\subsection{The weak and strong $M_{1}$ topologies}\label{ss:j1m1}

The stochastic processes that we consider have discontinuities, and hence for the function space of their sample paths we take the space $D([0,1], \mathbb{R}^{2})$ of all right-continuous $\mathbb{R}^{2}$--valued functions on $[0,1]$ with left limits.

 The stochastic processes $V_{n}$ and $W_{n}$ converge (separately) in the space $D([0,1], \mathbb{R})$ equipped with the standard $M_{1}$ topology, see Basrak et al.~\cite{BKS} and Krizmani\'{c}~\cite{Kr14}. In this paper we use the weak $M_{1}$ topology, since as we show later the functional convergence for $L_{n}$ in general fails to hold in the standard $M_{1}$ topology on $D([0,1], \mathbb{R}^{2})$. In the sequel we give the definitions of the weak and standard (strong) $M_{1}$ topologies.

 For $x \in D([0,1],
\mathbb{R}^{2})$ the completed graph of $x$ is the set
\[
  G_{x}
  = \{ (t,z) \in [0,1] \times \mathbb{R}^{2} : z \in [[x(t-), x(t)]]\},
\]
where $x(t-)$ is the left limit of $x$ at $t$ and $[[a,b]]$ is the product segment, i.e.
$[[a,b]]=[a_{1},b_{1}] \times [a_{2},b_{2}]$
for $a=(a_{1}, a_{2}), b=(b_{1}, b_{2}) \in
\mathbb{R}^{2}$. We define an
order on the graph $G_{x}$ by saying that $(t_{1},z_{1}) \le
(t_{2},z_{2})$ if either (i) $t_{1} < t_{2}$ or (ii) $t_{1} = t_{2}$
and $|x_{j}(t_{1}-) - z_{1j}| \le |x_{j}(t_{2}-) - z_{2j}|$
for all $j=1, 2$. Note that the relation $\le$ induces only a partial
order on the graph $G_{x}$. A weak parametric representation
of the graph $G_{x}$ is a continuous nondecreasing function $(r,u)$
mapping $[0,1]$ into $G_{x}$, with $r \in C([0,1],[0,1])$ being the
time component and $u \in C([0,1],
\mathbb{R}^{2})$ being the spatial component, such that $r(0)=0,
r(1)=1$ and $u(1)=x(1)$. Let $\Pi_{w}(x)$ denote the set of weak
parametric representations of the graph $G_{x}$. For $x_{1},x_{2}
\in D([0,1], \mathbb{R}^{2})$ define
\[
  d_{w}(x_{1},x_{2})
  = \inf \{ \|r_{1}-r_{2}\|_{[0,1]} \vee \|u_{1}-u_{2}\|_{[0,1]} : (r_{i},u_{i}) \in \Pi_{w}(x_{i}), i=1,2 \},
\]
where $\|x\|_{[0,1]} = \sup \{ \|x(t)\| : t \in [0,1] \}$. Now we
say that $x_{n} \to x$ in $D([0,1], \mathbb{R}^{2})$ for a sequence
$(x_{n})$ in the weak Skorohod $M_{1}$ (or shortly $WM_{1}$)
topology if $d_{w}(x_{n},x)\to 0$ as $n \to \infty$.

Now we recall the definition of the standard $M_{1}$ topology. For $x \in D([0,1], \mathbb{R}^{2})$
let
\[
  \Gamma_{x}
  = \{ (t,z) \in [0,1] \times \mathbb{R}^{2} : z \in [x(t-), x(t)] \},
\]
where $[a,b] = \{  \lambda a + (1-\lambda)b : 0 \leq \lambda \leq 1 \}$ for $a, b \in \mathbb{R}^{2}$. We say $(r,u)$ is a parametric representation of $\Gamma_{x}$ if it is a continuous nondecreasing function mapping $[0,1]$ onto $\Gamma_{x}$. Denote by $\Pi(x)$ the set of all parametric representations of the graph $\Gamma_{x}$. Then for $x_{1},x_{2} \in D([0,1], \mathbb{R}^{2})$ put
\[
  d_{M_{1}}(x_{1},x_{2})
  = \inf \{ \|r_{1}-r_{2}\|_{[0,1]} \vee \|u_{1}-u_{2}\|_{[0,1]} : (r_{i},u_{i}) \in \Pi(x_{i}), i=1,2 \}.
\]
$d_{M_{1}}$ is a metric on $D([0,1], \mathbb{R}^{2})$, and the induced topology is called the (standard or strong) Skorohod $M_{1}$ topology.
The $WM_{1}$
topology is weaker than the standard $M_{1}$ topology on $D([0,1],
\mathbb{R}^{2})$. The $WM_{1}$ topology
coincides with the topology induced by the metric
\begin{equation}\label{e:defdp}
 d_{p}(x_{1},x_{2})= \max \{ d_{M_{1}}(x_{1j},x_{2j}) : j=1,2 \}
\end{equation}
 for $x_{i}=(x_{i1}, x_{i2}) \in D([0,1],
 \mathbb{R}^{2})$ and $i=1,2$. The metric $d_{p}$ induces the product topology on $D([0,1], \mathbb{R}^{2})$.
For detailed discussion of the strong and weak $M_{1}$ topologies we refer to
Whitt~\cite{Whitt02}, sections 12.3--12.5.

\section{Functional convergence of $L_{n}$}\label{s:three}

In this section we show the convergence of the joint partial sum and maxima
process
\begin{equation*}
  L_{n}(t) = (V_{n}(t), W_{n}(t)), \quad t \in [0,1],
\end{equation*}
 in the space $D([0,1], \mathbb{R}^{2})$
equipped with Skorohod weak $M_1$ topology. We identify the limit as $(V, W)$, where $V$ is a stable L\'{e}vy process and $W$ an extremal process. First we represent $L_n$ as the image of the
time-space point process $N_n$ under an appropriate sum-maximum
functional. Then, using certain continuity properties of this
functional, by the continuous mapping theorem we
transfer the weak convergence of $N_n$ in (\ref{e:BaTa}) to
weak convergence of $L_n$.

\subsection{The sum-maximum functional}

Fix $0 < u < \infty$ and define the sum-maximum functional
$$ \Phi^{(u)} \colon \mathbf{M}_{p}([0,1] \times \EE) \to D([0,1], \mathbb{R}^{2})$$
by
$$ \Phi^{(u)} \Big( \sum_{i}\delta_{(t_{i}, x_{i})} \Big) (t)
  =  \Big( \sum_{t_{i} \leq t}x_{i}\,1_{\{u < |x_{i}| < \infty \}},  \bigvee_{t_{i} \leq t} x_{i} \vee 0 \Big), \qquad t \in [0,1].$$
The space $\mathbf{M}_p([0,1] \times \EE)$ of Radon point
measures on $[0,1] \times \EE$ is equipped with the vague
topology and $D([0,1], \mathbb{R}^{2})$ is equipped with the weak $M_1$ topology. For convenience we set $\sup \emptyset = 0$. Let $\Lambda = \Lambda_{1} \cap \Lambda_{2}$, where
\begin{multline*}
 \Lambda_{1} =
 \{ \eta \in \mathbf{M}_{p}([0,1] \times \EE) :
   \eta ( \{0,1 \} \times \EE) = 0 = \eta ([0,1] \times \{ \pm \infty, \pm u \}) \}, \\[1em]
 \shoveleft \Lambda_{2} =
 \{ \eta \in \mathbf{M}_{p}([0,1] \times \EE) :
  \eta ( \{ t \} \times (u, \infty]) \cdot \eta ( \{ t \} \times [-\infty,-u)) = 0 \\
  \text{for all $t \in [0,1]$} \}.
\end{multline*}
 Observe that the elements
of $\Lambda_2$ have the property that atoms in $[0,1] \times \mathbb{E}_{u}$ with the same time
coordinate are all on the same side of the time axis.
Similar to Lemma 3.1 in Basrak et al.~\cite{BKS} one can prove the following result.

\begin{lem}
\label{l:prob1}
Assume that with probability one the tail process $(Y_{i})_{i \in \mathbb{Z}}$ in (\ref{e:tailprocess}) has no two values of the opposite sign. Then
$ \PP ( N \in \Lambda ) = 1$.
\end{lem}

 Now we will show that
$\phi^{(u)}$ is continuous on the set $\Lambda$.

\begin{lem}\label{l:contfunct}
The sum-maximum functional $\Phi^{(u)} \colon \mathbf{M}_{p}([0,1]
\times \EE) \to D([0,1], \mathbb{R}^{2})$ is continuous on the set $\Lambda$,
when $D([0,1], \mathbb{R}^{2})$ is endowed with the weak $M_{1}$ topology.
\end{lem}

\begin{proof}
Take an arbitrary $\zeta \in \Lambda$ and suppose that $\zeta_{n} \vto \zeta$ in $\mathbf{M}_p([0,1] \times
\EE)$. We need to show that
$\Phi^{(u)}(\zeta_n) \to \Phi^{(u)}(\zeta)$ in $D([0,1],
\mathbb{R}^{2})$ according to the $WM_1$ topology. By
Theorem~12.5.2 in Whitt~\cite{Whitt02}, it suffices to prove that,
as $n \to \infty$,
$$ d_{p}(\Phi^{(u)}(\zeta_{n}), \Phi^{(u)}(\zeta)) =
\max_{k=1,2}d_{M_{1}}(\Phi^{(u)}_{k}(\zeta_{n}),
\Phi^{(u)}_{k}(\zeta)) \to 0.$$
Now one can follow, with small modifications, the lines in the proof of Lemma~3.2 in Basrak et al.~\cite{BKS} to obtain
$d_{M_{1}}(\Phi^{(u)}_{1}(\zeta_{n}), \Phi^{(u)}_{1}(\zeta)) \to 0$ as
$n \to \infty$.

Let
$$T= \{ t \in [0,1] : \zeta (\{t\} \times \EE) = 0 \}.$$
Since $\zeta$ is a Radon point measure, the set $T$ is dense in $[0,1]$. Fix $t \in T$ and take $\epsilon >0$ such that $\zeta([0,t] \times \{\epsilon\})=0$.
Later, when $\epsilon \downarrow 0$, we assume convergence to $0$ is through a sequence of values $(\epsilon_{j})$ such that $\zeta([0,t] \times \{\epsilon_{j}\})=0$ for all $j \in \mathbb{N}$ (this can be arranged since $\zeta$ is a Radon point measure). Since the set $[0,t] \times [\epsilon, \infty]$ is relatively compact in
$[0,1] \times \EE$, there exists a nonnegative integer
$k=k(\zeta)$ such that
$$ \zeta ([0,t] \times [\epsilon, \infty]) = k < \infty.$$
By assumption, $\zeta$ does not have any atoms on the border of the
set $[0,t] \times [\epsilon, \infty]$. Hence, by Lemma 7.1 in Resnick~\cite{Re07}, there exists a positive integer $n_{0}$ such
that for all $n \geq n_{0}$ it holds that
$$ \zeta_{n} ([0,t] \times [\epsilon, \infty])=k.$$
Let
$(t_{i},x_{i})$ for $i=1,\ldots,k$ be the atoms of $\zeta$ in
$[0,t] \times [\epsilon, \infty]$. By the same lemma, the $k$ atoms
$(t_{i}^{(n)}, x_{i}^{(n)})$ of $\zeta_{n}$ in $[0,t] \times [\epsilon, \infty]$ (for $n \geq n_{0}$) can be labeled in such a way that
for every $i \in \{1,\ldots,k\}$ we have
$$ (t_{i}^{(n)}, x_{i}^{(n)}) \to (t_{i},x_{i}) \qquad \textrm{as}
\ n \to \infty.$$ In particular, for any $\delta >0$ we can find a
positive integer $n_{\delta} \geq n_{0}$ such that for all $n \geq
n_{\delta}$,
\begin{equation*}\label{e:etaconv}
 |t_{i}^{(n)} - t_{i}| < \delta \quad \textrm{and} \quad
 |x_{i}^{(n)}- x_{i}| < \delta \qquad \textrm{for} \ i=1,\ldots,k.
\end{equation*}
If $k=0$, then (for large $n$) the atoms of $\zeta$ and $\zeta_{n}$ in $[0,t] \times \EE$ are all situated in $[0,t] \times [-\infty, \epsilon)$. Hence
$ \Phi^{(u)}_{2}(\zeta)(t) \in [0,\epsilon)$ and $ \Phi^{(u)}_{2}(\zeta_{n})(t) \in [0, \epsilon)$, which imply
\begin{equation}\label{e:conv1}
  |\Phi^{(u)}_{2}(\zeta_{n})(t) - \Phi^{(u)}_{2}(\zeta)(t)| < \epsilon.
\end{equation}
If $k \geq 1$, take $\delta = \epsilon$. Then we have
\begin{equation}\label{e:conv2}
  |\Phi^{(u)}_{2}(\zeta_{n})(t) - \Phi^{(u)}_{2}(\zeta)(t)| = \bigg| \bigvee_{i=1}^{k}x_{i}^{(n)} - \bigvee_{i=1}^{k}x_{i} \bigg| \leq \bigvee_{i=1}^{k}|x_{i}^{(n)}-x_{i}| <  \epsilon.
\end{equation}
Therefore form (\ref{e:conv1}) and (\ref{e:conv2}) we obtain
 $$\lim_{n \to \infty}|\Phi^{(u)}_{2}(\zeta_{n})(t) -
 \Phi^{(u)}_{2}(\zeta)(t)|< \epsilon,$$
  and if we let $\epsilon \to 0$, it follows that
 $\Phi^{(u)}_{2}(\zeta_{n})(t) \to \Phi^{(u)}_{2}(\zeta)(t)$ as $n \to
 \infty$. Note that $\Phi^{(u)}_{2}(\zeta)$ and $\Phi^{(u)}_{2}(\zeta_{n})$ are nondecreasing functions. Since $M_{1}$ convergence for monotone functions is equivalent to pointwise convergence in a dense subset of points plus convergence at the endpoints (Whitt~\cite{Whitt02}, Corollary 12.5.1), we conclude that $d_{M_{1}}(\Phi^{(u)}_{2}(\zeta_{n}),
 \Phi^{(u)}_{2}(\zeta)) \to 0$ as $n \to \infty$. Hence
$\Phi^{(u)}$ is continuous at $\zeta$.
\end{proof}

\smallskip

\smallskip

\subsection{Main theorem}

Let $(X_{n})$ be a strictly stationary sequence of random variables, regularly varying with index $\alpha \in (0,2)$. The theorem below gives conditions under which the joint partial sum and maxima
process $L_{n}$ satisfies a functional limit theorem with the limit $L = (V, W)$, where $V$ is an $\alpha$--stable L\'{e}vy process and $W$ is an extremal process.

The distribution of a L\'{e}vy process $V$ is characterized by its
characteristic triple, that is, the characteristic triple of the infinitely divisible distribution
of $V(1)$. The characteristic function of $V(1)$ and the characteristic triple
$(a, \sigma, b)$ are related in the following way:
\begin{equation}\label{e:Kintchin}
  \mathrm{E} [e^{izV(1)}] = \exp \biggl( -\frac{1}{2}az^{2} + ibz + \int_{\mathbb{R}} \bigl( e^{izx}-1-izx 1_{[-1,1]}(x) \bigr)\,\sigma(\rmd x) \biggr)
\end{equation}
for $z \in \mathbb{R}$. Here $a \ge 0$, $b \in \mathbb{R}$ are constants, and $\sigma$ is a measure on $\mathbb{R}$ satisfying
$$ \sigma ( \{0\})=0 \qquad \text{and} \qquad \int_{\mathbb{R}}(|x|^{2} \wedge 1)\,\sigma(\rmd x) < \infty.$$

\begin{rem}\label{r:stablechar}
 For $\alpha \in (0,2)$ it is sometimes convenient to rewrite
 the characteristic function of an $\alpha$--stable random variable $Y$ in the following form
\begin{equation}\label{e:charstab1}
  \mathrm{E}[e^{izY}] = \left\{ \begin{array}{lc}
       \exp \Big[ -c|z|^{\alpha} \big( 1 - i \beta \mathrm{sign}(z) \tan \frac{\pi \alpha}{2} \big)
           + i \tau z \Big], & \alpha \neq 1,\\[0.7em]
       \exp \Big[ -c|z| \big( 1 + i \beta \frac{2}{\pi} \mathrm{sign}(z) \log |z| \big)
       + i \tau z \Big], & \alpha =1,
                                 \end{array}\right.
\end{equation}
where $c>0$, $\beta \in [-1,1]$ and $\tau \in \mathbb{R}$ (Sato~\cite{Sa99}, Theorem 14.15). The representations of the characteristic function of a stable
distribution in the L\'{e}vy-Khintchine representation
(\ref{e:Kintchin}) and relation (\ref{e:charstab1}) are connected in the following way:
$$ a=0, \quad \sigma(dx) = \big( c_{1} 1_{(0,\infty)}(x) + c_{2}
1_{(-\infty,0)}(x) \big) \, |x|^{-1-\alpha}\,\rmd x \quad \textrm{and}
\quad b= \tau - d,$$ where
$$\begin{array}{lllll}
(i) & c_{1} = \frac{-c(1+\beta)}{2 \Gamma(-\alpha) \cos(\pi
  \alpha /2)}, & c_{2}=\frac{-c(1-\beta)}{2 \Gamma(-\alpha) \cos(\pi
  \alpha /2)}, & d=-\int_{|x| \leq 1}x\,\sigma(\rmd x), & \ \textrm{if} \ \alpha < 1;\\[0.8em]
(ii) & c_{1} = \frac{-c(1+\beta)}{2 \Gamma(-\alpha) \cos(\pi
  \alpha /2)}, & c_{2}=\frac{-c(1-\beta)}{2 \Gamma(-\alpha) \cos(\pi
  \alpha /2)}, & d=\int_{|x| > 1}x\,\sigma(\rmd x), & \ \textrm{if} \ \alpha > 1;\\[0.8em]
(iii) &  c_{1}=\frac{c(1+\beta)}{\pi}, &  c_{2}=\frac{c(1-\beta)}{\pi}, & d = (c_{1}-c_{2})s, & \ \textrm{if} \ \alpha=1;
\end{array}$$
with $ s = \int_{0}^{\infty}(\sin x - x1_{(0,1]}(x))x^{-2}\rmd x$
(see Lemma 2 in Feller~\cite{Feller71}, p.~541, and Theorem 14.3 and Lemma 14.11 in
Sato~\cite{Sa99}).
\end{rem}

The distribution of a nonnegative extremal process $W$ is characterized by its exponent measure $\nu''$ in the following way:
$$ \PP (W(t) \leq x ) = e^{-t \nu''(x,\infty)}$$
for $t>0$ and $x>0$, where $\nu''$ is a measure on $(0,\infty)$ satisfying
$ \nu'' (\delta, \infty) < \infty$
for any $\delta >0$ (see Resnick~\cite{Re07}, page 161).

The description of the characteristic triple of $V$ and the exponent measure of $W$ in the limiting process will be in terms of the measures $\nu'$ and $\nu''$ on $\mathbb{R}$ defined by
$$ \nu'(\rmd x) = \big( c_{+} 1_{(0,\infty)}(x) + c_{-} 1_{(-\infty, 0)}(x) \big) \theta \alpha |x|^{-\alpha -1}\,\rmd x$$
and
$$ \nu''(\rmd x) = r  \theta \alpha x^{-\alpha -1} 1_{(0,\infty)}(x)\,\rmd x,$$
where
$$ c_{+} = \mathrm{E} \bigg[ \Big( \sum_{j}\eta_{j} \Big)^{\alpha} 1_{ \{ \sum_{j}\eta_{j} > 0 \}} \bigg], \quad
   c_{-} = \mathrm{E} \bigg[ \Big( - \sum_{j}\eta_{j} \Big)^{\alpha} 1_{ \{ \sum_{j}\eta_{j} < 0 \}} \bigg]$$
and
$$ r = \mathrm{E} \Big( \bigvee_{j}\eta_{j} \vee 0 \Big)^{\alpha},$$
with $\theta$ and $(\eta_{j})_{j}$ as defined in Section~\ref{s:pp}.
When $\alpha \in [1,2)$ we need an additional assumption to deal with the small jumps.

\begin{cond}\label{c:step6cond}
For all $\delta > 0$,
$$
  \lim_{u \downarrow 0} \limsup_{n \to \infty} \PP \bigg[
  \max_{0 \le k \le n}  \bigg| \sum_{i=1}^{k} \bigg( \frac{X_{i}}{a_{n}}
  1_{ \big\{ \frac{|X_{i}|}{a_{n}} \le u \big\} } -  \E \bigg( \frac{X_{i}}{a_{n}}
  1_{ \big\{ \frac{|X_{i}|}{a_{n}} \le u \big\} } \bigg) \bigg) \bigg| > \delta
  \bigg]=0.$$
\end{cond}
Condition~\ref{c:step6cond} is not easily checked for dependent sequences, but for instance it holds for $\rho$--mixing processes with a certain rate (see Tyran-Kami\'{n}ska~\cite{TK2010}).
In the case $\alpha=1$ we will assume additionally
\begin{equation}\label{e:BPScond1}
 \mathrm{E} \bigg[ \sum_{j}\eta_{j} \log \Big( |\eta_{j}|^{-1} \sum_{i}|\eta_{i}| \Big) \bigg] < \infty,
\end{equation}
with the convention $\eta_{j} \log ( |\eta_{j}|^{-1} \sum_{i}|\eta_{i}|)=0$ if $\eta_{j}=0$  (for a discussion about condition (\ref{e:BPScond1}) see Basrak et al.~\cite{BaPlSo}, Remark 4.8).
Let
\begin{equation}\label{e:gamma}
 \gamma = \left\{ \begin{array}{ll}
                                   \frac{\theta \alpha}{1-\alpha} (c_{+}-c_{-}), & \quad \alpha \in (0,1),\\[0.6em]
                                   \frac{\alpha}{\alpha-1} \Big( p-q - \theta (c_{+}-c_{-}) \Big), & \quad \alpha \in (1,2),\\[0.6em]
                                   - \theta\mathrm{E}\bigg[ \sum_{j}\eta_{j} \textrm{log} \Big( \Big| \sum_{i}\eta_{i}\eta_{j}^{-1} \Big|  \Big) \bigg], & \quad \alpha =1,
                                 \end{array}\right.
\end{equation}
with $p$ and $q$ as in (\ref{e:mu}).

\begin{thm}\label{t:functconvergence}
Let $(X_{n})$ be a strictly stationary sequence of random variables, jointly regularly varying with index $\alpha\in(0,2)$, and of which the tail process $(Y_{i})_{i \in \mathbb{Z}}$ almost surely has no two values of the opposite sign. Suppose that Conditions~\ref{c:mixcond1} and~\ref{c:mixcond2} hold. If $\alpha \in [1,2)$ suppose also Condition~\ref{c:step6cond} and relation (\ref{e:MW1}) hold, and for $\alpha =1$ assume further that (\ref{e:BPScond1}) holds. Then the stochastic process
\begin{equation*}
  L_{n}(t) =
  \bigg( \sum_{k=1}^{[nt]} \frac{X_{k}}{a_{n}} -\floor{nt}b_{n}, \bigvee_{i=1}^{\lfloor nt \rfloor}\frac{X_{i}}{a_{n}} \bigg)
   \quad t \in [0,1],
\end{equation*}
with $a_{n}$ and $b_{n}$ as in (\ref{e:niz}) and (\ref{e:centb}),
satisfies
$$ L_{n} \dto L \qquad \textrm{as} \ n \to \infty,$$
in $D([0,1], \mathbb{R}^{2})$ endowed with the weak $M_{1}$ topology, where $L = (V, W)$, $V$
is an $\alpha$--stable L\'{e}vy process with characteristic triple $(0,
\nu', \gamma)$ and $W$ is an extremal process with exponent measure $\nu''$.
\end{thm}

\begin{proof}
Take an arbitrary $u>0$, and consider
$$ \Phi^{(u)}(N_{n})(\,\cdot\,) = \bigg( \sum_{i/n \leq\,\cdot}\frac{X_{i}}{a_{n}} 1_{ \big\{ \frac{|X_{i}|}{a_{n}} > u \big\} }, \bigvee_{i/n \leq\,\cdot}\frac{X_{i}}{a_{n}} \vee 0 \bigg).$$
From Lemma~\ref{l:prob1} and Lemma~\ref{l:contfunct} we know that $\Phi^{(u)}$ is continuous on the set $\Lambda$ and this set almost surely contains the limiting point process $N$ from (\ref{e:BaTa}). Hence an application of the continuous mapping theorem yields
$\Phi^{(u)}(N_{n}) \dto \Phi^{(u)}(N)$ in $D([0,1], \mathbb{R}^{2})$ under the weak $M_{1}$ topology, i.e.
\begin{equation}\label{e:mainconv0}
 \bigg( \sum_{i = 1}^{\lfloor n \, \cdot \, \rfloor} \frac{X_{i}}{a_{n}}
    1_{ \bigl\{ \frac{|X_{i}|}{a_{n}} > u \bigr\} }, \bigvee_{i=1}^{\lfloor n \, \cdot \, \rfloor}\frac{X_{i}}{a_{n}} \vee 0 \bigg)
    \dto   \bigg( \sum_{T_{i} \le \, \cdot} \sum_{j}P_{i}\eta_{ij} 1_{\{ P_{i}|\eta_{ij}| > u \}}, \bigvee_{T_{i} \le \, \cdot} \bigvee_{j}P_{i}\eta_{ij} \vee 0 \bigg).
 \end{equation}
From (\ref{e:onedimregvar}) we have, as $n \to \infty$,
\begin{eqnarray}\label{e:conv12}
 \nonumber \floor{nt} \mathrm{E} \Big( \frac{X_{1}}{a_{n}} 1_{ \big\{ u < \frac{|X_{1}|}{a_{n}} \leq 1 \big\} } \Big) &=& \frac{\floor{nt}}{n} \int_{u < |x| \leq 1} x n\,\PP \Big( \frac{X_{1}}{a_{n}} \in \rmd x \Big) \\[0.6em]
   & \to &  t \int_{u < |x| \leq 1} x\mu(\rmd x)
\end{eqnarray}
for every $t \in [0,1]$, and this convergence is uniform in $t$. Hence applying Lemma 2.1 from Krizmani\'{c}~\cite{Kr18} (adjusted for the $M_{1}$ convergence) to (\ref{e:mainconv0}) and (\ref{e:conv12}) we obtain, as $n \to \infty$,
\begin{multline}\label{e:mainconv}
      L_{n}^{(u)}(\,\cdot\,) := \bigg( \sum_{i = 1}^{\lfloor n \, \cdot \, \rfloor} \frac{X_{i}}{a_{n}}
    1_{ \bigl\{ \frac{|X_{i}|}{a_{n}} > u \bigr\} } - \floor{n\,\cdot\,}b_{n}^{(u)}, \bigvee_{i=1}^{\lfloor n \, \cdot \, \rfloor}\frac{X_{i}}{a_{n}} \vee 0 \bigg) \\
    \dto L^{(u)}(\,\cdot\,) :=  \bigg( \sum_{T_{i} \le \, \cdot} \sum_{j}P_{i}\eta_{ij} 1_{\{ P_{i}|\eta_{ij}| > u \}} - (\,\cdot\,) b^{(u)}, \bigvee_{T_{i} \le \, \cdot} \bigvee_{j}P_{i}\eta_{ij} \vee 0 \bigg),
\end{multline}
where
$$  b_{n}^{(u)} = \left\{ \begin{array}{cc}
                                   0, & \quad \alpha \in (0,1),\\[0.5em]
                                   \mathrm{E} \Big( \frac{X_{1}}{a_{n}} 1_{ \big\{ u < \frac{|X_{1}|}{a_{n}} \leq 1 \big\} } \Big), & \quad \alpha \in [1,2)
                                 \end{array}\right.$$
                                 and
$$ b^{(u)} = \left\{ \begin{array}{cc}
                                   0, & \quad \alpha \in (0,1),\\[0.6em]
                                    \int_{u < |x| \leq 1} x\mu(\rmd x), & \quad \alpha \in [1,2).
                                 \end{array}\right.$$
Now we split the rest of the proof into two cases: $\alpha \in (0,1)$ and $\alpha \in [1,2)$.

Assume first $\alpha \in (0,1)$.
Let $U_{i} = \sum_{j}\eta_{ij}$, $i=1,2,\ldots$.
Since $ \alpha < 1$, from (\ref{e:MW1}) we obtain
\begin{equation*}\label{e:MW}
 \mathrm{E}|U_{1}|^{\alpha} \leq \mathrm{E} \Big( \sum_{j}|\eta_{j}| \Big)^{\alpha} < \infty.
\end{equation*}
It is straightforward to check that $P_{i}|U_{i}|$, $i=1,2,\ldots$, are the points of a Poisson process with intensity measure $\theta \mathrm{E} |U_{1}|^{\alpha} \alpha x^{-\alpha-1}\,\rmd x$ for $x>0$ (see Proposition 5.2 and Proposition 5.3 in Resnick~\cite{Re07}). These points are summable, i.e.
 $$ \sum_{i=1}^{\infty}P_{i}|U_{i}| \leq \sum_{i=1}^{\infty}\sum_{j}P_{i}|\eta_{ij}| < \infty$$
 almost surely
 (see the proof of Theorem 3.1 in Davis and Hsing~\cite{DaHs95}), and therefore for all $t\in [0,1]$
$$ \sum_{T_{i} \leq t} \sum_{j} P_{i}\eta_{ij} 1_{\{ P_{i}|\eta_{ij}| > u \}} \to \sum_{T_{i} \leq t} \sum_{j}P_{i}\eta_{ij}$$
almost surely as $u \to 0$. By the dominated convergence theorem
$$ \sup_{t \in [0,1]} \bigg| \sum_{T_{i} \leq t} \sum_{j} P_{i}\eta_{ij} 1_{\{ P_{i}|\eta_{ij}| > u \}} - \sum_{T_{i} \leq t} \sum_{j}P_{i}\eta_{ij} \bigg| \leq \sum_{i=1}^{\infty}\sum_{j}P_{i}|\eta_{ij}| 1_{ \{ P_{i}|\eta_{ij}| \leq u \} } \to 0$$
almost surely as $u \to 0$.
Since uniform convergence implies Skorohod $M_{1}$ convergence, we get
\begin{equation}\label{e:dm1}
d_{M_{1}} \Big( \sum_{T_{i} \leq\,\cdot} \sum_{j} P_{i}\eta_{ij} 1_{\{ P_{i}|\eta_{ij}| > u \}}, \sum_{T_{i} \leq\,\cdot} \sum_{j}P_{i}\eta_{ij} \Big) \to 0
\end{equation}
almost surely as $u \to 0$.
Let
$$L(\,\cdot\,) = \bigg( \sum_{T_{i} \le \, \cdot} \sum_{j}P_{i}\eta_{ij}, \bigvee_{T_{i} \le \, \cdot} \bigvee_{j}P_{i}\eta_{ij} \vee 0 \bigg).$$
Recalling the definition of the metric $d_{p}$ in (\ref{e:defdp}), from (\ref{e:dm1}) we obtain
\begin{equation*}
d_{p}(L^{(u)}, L) \to 0
\end{equation*}
almost surely as $u \to 0$. Since almost sure convergence implies weak convergence, we have, as $u \to 0$,
\begin{equation}\label{e:mainconv2}
 L^{(u)} \dto L
\end{equation}
in $D([0,1], \mathbb{R}^{2})$ endowed with the weak $M_{1}$ topology.
By Proposition 5.2 and Proposition 5.3 in Resnick~\cite{Re07}, the process
$$\sum_{i} \delta_{(T_{i}, \sum_{j}P_{i}\eta_{ij} })$$
 is a Poisson process with intensity measure
$ Leb \times \nu'$.
Similarly, the process
$$\sum_{i} \delta_{(T_{i}, \bigvee_{j}P_{i}\eta_{ij} \vee 0})$$
 is a Poisson process with intensity measure
$ Leb \times \nu''$.
By the It\^{o} representation of the L\'{e}vy process (see Resnick~\cite{Re07}, pages 150--153) and Theorem 14.3 in Sato~\cite{Sa99},
$$ V (\,\cdot\,) = \sum_{T_{i} \leq\,\cdot} \sum_{j}P_{i}\eta_{ij}$$ is an $\alpha$--stable L\'{e}vy process with characteristic triple $(0, \nu', (c_{+}-c_{-}) \theta \alpha / (1-\alpha))$.
Also
$$ W (\,\cdot\,) = \bigvee_{T_{i} \leq\,\cdot} \bigvee_{j}P_{i}\eta_{ij} \vee 0$$
is an extremal process with exponent measure $\nu''$ (see Resnick~\cite{Re07}, page 161).

If we show that
$$ \lim_{u \to 0}\limsup_{n \to \infty} \PP(d_{p}(L_{n},L_{n}^{(u)}) > \epsilon)=0$$
for any $\epsilon >0$, from (\ref{e:mainconv}) and (\ref{e:mainconv2}) by a variant of Slutsky's theorem (see Theorem 3.5 in Resnick~\cite{Re07}) it will follow that
$ L_{n} \dto L$ as $n \to \infty$,
in $D([0,1], \mathbb{R}^{2})$ with the weak $M_{1}$ topology.
Since the
 metric $d_{p}$ on $D([0,1], \mathbb{R}^{2})$ is bounded above by the uniform metric on
 $D([0,1], \mathbb{R}^{2})$ (Whitt~\cite{Whitt02}, Theorem 12.10.3), it suffices to show that
 $$ \lim_{u \to 0} \limsup_{n \to \infty} \PP \biggl(
 \sup_{t \in [0,1]} \|L_{n}(t) - L_{n}^{(u)}(t)\| >
 \epsilon \biggr)=0.$$
 Using stationarity, Markov's inequality and the fact that
 $$\bigg| \bigvee_{i=1}^{\lfloor nt \rfloor}\frac{X_{i}}{a_{n}} - \bigvee_{i=1}^{ \lfloor nt \rfloor }\frac{X_{i}}{a_{n}} \vee 0 \bigg| \leq \frac{|X_{1}|}{a_{n}} 1_{\{ X_{1} < 0 \}},$$ we get the bound
 \begin{eqnarray}\label{e:slutsky}
    \nonumber \PP \bigg(
     \sup_{t \in [0,1]} \|L_{n}(t) - L_{n}^{(u)}(t)\| >  \epsilon \bigg) & & \\[0.3em]
   \nonumber & \hspace*{-20em} = & \hspace*{-10em} \PP \bigg(
       \sup_{t \in [0,1]} \ \max \bigg\{  \bigg| \sum_{i=1}^{\lfloor nt \rfloor} \frac{X_{i}}{a_{n}}
       1_{ \big\{ \frac{|X_{i}|}{a_{n}} \leq u \big\} }  \bigg|, \frac{|X_{1}|}{a_{n}} 1_{\{ X_{1} < 0 \}} \bigg\}  > \epsilon
       \bigg)\\[0.4em]
   \nonumber  & \hspace*{-20em} \leq & \hspace*{-10em} \PP \bigg(
       \sum_{i=1}^{n} \frac{|X_{i}|}{a_{n}}
       1_{ \big\{ \frac{|X_{i}|}{a_{n}} \leq u \big\} }  > \epsilon
       \bigg) + \PP \bigg( \frac{|X_{1}|}{a_{n}} 1_{\{ X_{1} < 0 \}} > \epsilon
       \bigg)\\[0.4em]
      & \hspace*{-20em} \leq & \hspace*{-10em}
      \epsilon^{-1} n \mathrm{E} \bigg( \frac{|X_{1}|}{a_{n}}
       1_{ \big\{ \frac{|X_{1}|}{a_{n}} \leq u \big\} } \bigg) + \PP \bigg( \frac{|X_{1}|}{a_{n}} 1_{\{ X_{1} < 0 \}} > \epsilon
       \bigg).
 \end{eqnarray}
 For the first term on the right-hand side of (\ref{e:slutsky}) we have
 $$n \mathrm{E} \bigg( \frac{|X_{1}|}{a_{n}}
       1_{ \big\{ \frac{|X_{1}|}{a_{n}} \leq u \big\} } \bigg) =  u \cdot n \PP (|X_{1}| > a_{n}) \cdot \frac{\PP(|X_{1}| > ua_{n})}{\PP(|X_{1}|>a_{n})} \cdot
          \frac{\mathrm{E}(|X_{1}| 1_{ \{ |X_{1}| \leq ua_{n} \} })}{ ua_{n} \PP (|X_{1}| >ua_{n})}.$$
 Since $X_{1}$ is a regularly varying random variable with index $\alpha$, it follows immediately that
  $$ \frac{\PP(|X_{1}| > ua_{n})}{\PP(|X_{1}|>a_{n})} \to u^{-\alpha}$$
  as $n \to \infty$. By Karamata's theorem
  $$ \lim_{n \to \infty} \frac{\mathrm{E}(|X_{1}| 1_{ \{ |X_{1}| \leq ua_{n} \} })}{ ua_{n} \PP (|X_{1}| >ua_{n})} = \frac{\alpha}{1-\alpha}.$$
 Therefore, taking into account relation (\ref{e:niz}), we get
 $$ n \mathrm{E} \bigg( \frac{|X_{1}|}{a_{n}}
       1_{ \big\{ \frac{|X_{1}|}{a_{n}} \leq u \big\} } \bigg) \to u^{1-\alpha} \frac{\alpha}{1-\alpha}$$
 as $n \to \infty$. Observe that
 $$ \frac{|X_{1}|}{a_{n}} 1_{\{X_{1} < 0\}} \to 0$$
 almost surely as $n \to \infty$, and thus
 \begin{equation}\label{e:onevariable}
   \PP \bigg( \frac{|X_{1}|}{a_{n}} 1_{\{ X_{1} < 0 \}} > \epsilon
       \bigg) \to 0
 \end{equation}
 as $n \to \infty$. Therefore from (\ref{e:slutsky}) we obtain
 $$ \limsup_{n \to \infty} \PP \bigg(
     \sup_{t \in [0,1]} \|L_{n}(t) - L_{n}^{(u)}(t)\| >  \epsilon \bigg) \leq \epsilon^{-1} u^{1-\alpha} \frac{\alpha}{1-\alpha}.$$
 Letting $u \to 0$, since $1-\alpha >0$, we finally obtain
 $$ \lim_{u \to 0} \limsup_{n \to \infty} \PP \bigg(
     \sup_{t \in [0,1]} \|L_{n}(t) - L_{n}^{(u)}(t)\| >  \epsilon \bigg) = 0.$$

Assume now $\alpha \in [1,2)$. Since (\ref{e:MW1}) holds by assumption in this case, by Lemma 6.4 in Basrak et al.~\cite{BaPlSo} there exists an $\alpha$--stable L\'{e}vy process $V$ such that, as $u \to 0$ (along some subsequence)
$$ \sum_{T_{i} \le \, \cdot} \sum_{j}P_{i}\eta_{ij} 1_{\{ P_{i}|\eta_{ij}| > u \}} - (\,\cdot\,) b^{(u)} \to V(\,\cdot\,)$$
uniformly almost surely.
As in the case $\alpha \in (0,1)$ we obtain, as $u \to 0$,
\begin{equation}\label{e:mainconv2-2}
 L^{(u)} \dto L
\end{equation}
in $D([0,1], \mathbb{R}^{2})$ endowed with the weak $M_{1}$ topology, where
$$L(\,\cdot\,) = \bigg( \lim_{u \to 0} \Big( \sum_{T_{i} \le \, \cdot} \sum_{j}P_{i}\eta_{ij} 1_{\{ P_{i}|\eta_{ij}| > u \}} - (\,\cdot\,) b^{(u)} \Big), \bigvee_{T_{i} \le \, \cdot} \bigvee_{j}P_{i}\eta_{ij} \vee 0 \bigg).$$
As before, we need to show that
 $$ \lim_{u \to 0} \limsup_{n \to \infty} \PP \biggl(
 \sup_{t \in [0,1]} \|L_{n}(t) - L_{n}^{(u)}(t)\| >
 \epsilon \biggr)=0,$$
 since then by the Slutsky argument it will follow $ L_{n} \dto L$ as $n \to \infty$,
in $D([0,1], \mathbb{R}^{2})$ with the weak $M_{1}$ topology. Recalling the definitions of $L_{n}$ and $L_{n}^{(u)}$, we have
\begin{eqnarray}\label{e:slutsky-2}
    \nonumber \PP \bigg(
     \sup_{t \in [0,1]} \|L_{n}(t) - L_{n}^{(u)}(t)\| >  \epsilon \bigg) & & \\[0.3em]
   \nonumber & \hspace*{-27em} = & \hspace*{-13.5em} \PP \bigg(
       \sup_{t \in [0,1]} \ \max \bigg\{  \bigg| \sum_{i=1}^{\lfloor nt \rfloor} \frac{X_{i}}{a_{n}}
       1_{ \big\{ \frac{|X_{i}|}{a_{n}} \leq u \big\} } - \lfloor nt \rfloor  \mathrm{E} \Big( \frac{X_{1}}{a_{n}} 1_{ \big\{  \frac{|X_{1}|}{a_{n}} \leq u \big\} } \Big) \bigg|, \frac{|X_{1}|}{a_{n}} 1_{\{ X_{1} < 0 \}} \bigg\}  > \epsilon
       \bigg)\\[0.4em]
   \nonumber & \hspace*{-27em} \leq & \hspace*{-13.5em} \PP \bigg( \max_{0 \leq k \leq n}
       \bigg| \sum_{i=1}^{k} \bigg( \frac{X_{i}}{a_{n}}
       1_{ \big\{ \frac{|X_{i}|}{a_{n}} \leq u \big\} } - \mathrm{E} \bigg( \frac{X_{i}}{a_{n}}
       1_{ \big\{ \frac{|X_{i}|}{a_{n}} \leq u \big\} } \bigg) \bigg) \bigg| > \epsilon
       \bigg) + \PP \bigg( \frac{|X_{1}|}{a_{n}} 1_{\{ X_{1} < 0 \}} > \epsilon
       \bigg).
 \end{eqnarray}
 Therefore, from Condition~\ref{c:step6cond} and (\ref{e:onevariable}) it follows
 $$ \lim_{u \to 0} \limsup_{n \to \infty} \PP \bigg(
     \sup_{t \in [0,1]} \|L_{n}(t) - L_{n}^{(u)}(t)\| >  \epsilon \bigg) = 0.$$
To conclude the proof it remains to show that $(0,\nu', \gamma)$ is the characteristic triple of the process $V$. According to Basrak et al.~\cite{BaPlSo}, Remark 4.7 (see also Davis and Hsing~\cite{DaHs95}, Theorem 3.2), the characteristic function of $V(1)$ is of the form
$$ \mathrm{E}[e^{izV(1)}] = \left\{ \begin{array}{lc}
                                   \mathrm{exp} \Big[ -c|z|^{\alpha} \Big( 1 - i\beta \textrm{sign}(z) \textrm{tan} \frac{\pi \alpha}{2} \Big) + i\tau z \Big], & \alpha > 1,\\[0.7em]
                                    \mathrm{exp} \Big[ -c|z| \Big( 1 + i\beta \frac{2}{\pi} \textrm{sign}(z) \textrm{log}|z| \Big) + i\tau z \Big], & \alpha =1,
                                 \end{array}\right.$$
where, with the notation $x^{\langle \alpha \rangle} = x |x|^{\alpha-1} = |x|^{\alpha} ( 1_{(0,\infty)}(x) - 1_{(-\infty,0)}(x))$,
\begin{equation}\label{e:characttr1}
 c= \theta \mathrm{E} \bigg[ \Big| \sum_{j} \eta_{j} \Big|^{\alpha} \bigg], \qquad \beta = \frac{\mathrm{E}[( \sum_{j}\eta_{j})^{\langle \alpha \rangle}]}{\mathrm{E}[|\sum_{j}\eta_{j}|^{\alpha}]}
\end{equation}
and
\begin{equation}\label{e:characttr2}
 \tau = \left\{ \begin{array}{lc}
                                    (\alpha-1)^{-1} \alpha \theta \mathrm{E} \Big[ \sum_{j}\eta_{j}^{\langle \alpha \rangle} \Big], & \alpha >1,\\[0.7em]
                                   \theta \Big(s \mathrm{E} \Big[ \sum_{j}\eta_{j} \Big] - \mathrm{E}\Big[ \sum_{j}\eta_{j} \textrm{log} \big( \big| \sum_{i}\eta_{i}\eta_{j}^{-1} \big|  \big) \Big] \Big), & \alpha =1,
                                 \end{array}\right.
\end{equation}
with $ s = \int_{0}^{\infty}(\sin x - x1_{(0,1]}(x))x^{-2}\rmd x$. Then, as noted in Remark~\ref{r:stablechar}, the characteristic triple of the process $V$ is of the form $(0, \sigma, b)$, where
$\sigma(\rmd x)=\big( c_{1} 1_{(0,\infty)}(x) + c_{2}1_{(-\infty,0)}(x) \big)|x|^{-\alpha-1}\rmd x$ and $b=\tau-d$, with
$$c_{1} = \left\{ \begin{array}{ll}
 \frac{-c(1+\beta)}{2 \Gamma(-\alpha) \cos(\pi
  \alpha /2)}, & \alpha > 1,\\[0.7em]
  \frac{c(1+\beta)}{\pi}, & \alpha=1,
\end{array}\right.
\qquad
c_{2} = \left\{ \begin{array}{ll}
 \frac{-c(1-\beta)}{2 \Gamma(-\alpha) \cos(\pi
  \alpha /2)}, & \alpha > 1,\\[0.7em]
  \frac{c(1-\beta)}{\pi}, & \alpha=1
\end{array}\right.$$
and
$$ d = \left\{ \begin{array}{ll}
  \int_{|x| > 1}x\,\sigma(\rmd x), & \alpha > 1,\\[0.7em]
  (c_{1}-c_{2})s, & \alpha=1.
\end{array}\right.$$
Hence taking into consideration relations (\ref{e:characttr1}) and (\ref{e:characttr2}), by standard computations we obtain
$$ c_{1}= \alpha \theta \mathrm{E} \bigg[ \Big| \sum_{j}\eta_{j} \Big|^{\alpha} 1_{ \{ \sum_{j}\eta_{j} > 0 \}} \bigg] = \alpha \theta c_{+}, \quad
   c_{2} = \alpha \theta \mathrm{E} \bigg[ \Big| \sum_{j}\eta_{j} \Big|^{\alpha} 1_{ \{ \sum_{j}\eta_{j} < 0 \}} \bigg] = \alpha \theta c_{-},$$
and
$$ d= \left\{ \begin{array}{ll}
  \frac{c_{1}-c_{2}}{\alpha -1} = \frac{\theta \alpha}{\alpha -1} (c_{+}-c_{-}), & \alpha > 1\\[0.7em]
  (c_{1}-c_{2})s = \theta (c_{+}-c_{-})s , & \alpha=1
\end{array}\right..$$
Therefore $\sigma=\nu'$. For $\alpha \in (1,2)$, as shown in the proof of Theorem 3.2 in Davis and Hsing~\cite{DaHs95}, it holds that
$\theta \mathrm{E} \Big[ \sum_{j}\eta_{j}^{\langle \alpha \rangle} \Big]=p-q$, and this implies
$$ b = \tau-d = \frac{\alpha}{\alpha-1} \Big( p-q - \theta (c_{+}-c_{-}) \Big).$$
Note that $\mathrm{E} \Big[ \sum_{j}\eta_{j} \Big] = c_{+}-c_{-}$ for $\alpha=1$, which yields
$$b=\tau - d = - \theta\mathrm{E}\bigg[ \sum_{j}\eta_{j} \textrm{log} \Big( \Big| \sum_{i}\eta_{i}\eta_{j}^{-1} \Big|  \Big) \bigg].$$
Hence $b = \gamma$, and thus the L\'{e}vy process $V$ has characteristic triple $(0,\nu',\gamma)$.
\end{proof}

\begin{rem}\label{r:jointdepend}
From the proof of Theorem~\ref{t:functconvergence} it follows that the components of the limiting process $L=(V, W)$ can be expressed as functionals of the limiting point process $N = \sum_{i} \sum_{j} \delta_{(T_{i}, P_{i}\eta_{ij})}$ from relation (\ref{e:BaTa}), i.e.
$$ V(\,\cdot\,) =  \left\{ \begin{array}{ll}
  \sum_{T_{i} \leq\,\cdot} \sum_{j}P_{i}\eta_{ij}, & \quad \alpha \in (0,1),\\[0.7em]
  \lim_{u \to 0} \Big( \sum_{T_{i} \le \, \cdot} \sum_{j}P_{i}\eta_{ij} 1_{\{ P_{i}|\eta_{ij}| > u \}} - (\,\cdot\,)  \int_{u < |x| \leq 1} x\mu(\rmd x) \Big), & \quad \alpha \in [1,2),
\end{array}\right.$$
where the limit in the latter case holds almost surely uniformly on $[0,1]$ (along some subsequence), and
$$ W(\,\cdot\,) = \bigvee_{T_{i} \leq\,\cdot} \bigvee_{j}P_{i}\eta_{ij} \vee 0.$$
\end{rem}

\begin{rem}\label{r:j1m1}
The weak $M_{1}$ convergence in Theorem~\ref{t:functconvergence} in
general can not be replaced by the standard $M_{1}$ convergence. This is
shown in Example~\ref{ex:WM1M1}.
The problem in our proof if we consider the standard $M_{1}$ topology is Lemma~\ref{l:contfunct}, which in this case does not hold.
To see this, fix $u>0$ and define
  $$ \zeta_{n} = \delta_{(\frac{1}{2}- \frac{1}{n}, \frac{u}{2})} + \delta_{(\frac{1}{2}, 2u)} \qquad \textrm{for} \  n \geq 3.$$
 Then $\zeta_{n} \vto \zeta$, where
  $$\zeta = \delta_{(\frac{1}{2}, \frac{u}{2})} + \delta_{(\frac{1}{2}, 2u)} \in \Lambda.$$
 It is easy to compute
 $$ \Phi^{(u)}_{1}(\zeta_{n})(t) = 2u\,1_{[\frac{1}{2}, 1]}(t) \quad \textrm{and} \quad \Phi^{(u)}_{2}(\zeta_{n})(t) = \frac{u}{2}\,1_{[\frac{1}{2} - \frac{1}{n}, \frac{1}{2}]}(t) + 2u\,1_{[\frac{1}{2}, 1]}(t).$$
 Then
 $$ y_{n} (t) := \Phi^{(u)}_{2}(\zeta_{n})(t) - \Phi^{(u)}_{1}(\zeta_{n})(t) = \frac{u}{2}\,1_{[\frac{1}{2} - \frac{1}{n}, \frac{1}{2})}(t), \quad t \in [0,1],$$
 and similarly
 $$ y(t) := \Phi^{(u)}_{2}(\zeta)(t) - \Phi^{(u)}_{1}(\zeta)(t)=0, \quad t \in [0,1].$$
 Hence $d_{M_{1}}(y_{n}, y) \geq u/2$ for all $n \geq 3$, which means that $d_{M_{1}}(y_{n}, y)$ does not converge to zero as $n \to \infty$.
 Since
 $$ d_{M_{1}}(y_{n}, y) \leq d_{M_{1}}(\Phi^{(u)}(\zeta_{n}), \Phi^{(u)}(\zeta))$$
 (Whitt~\cite{Whitt02}, Theorem 12.7.1),
 we conclude that $d_{M_{1}}(\Phi^{(u)}(\zeta_{n}), \Phi^{(u)}(\zeta))$ does not converge to zero. Therefore the functional $\Phi^{(u)}$ is not continuous at $\zeta$ with respect to the standard $M_{1}$ topology.
\end{rem}

\section{Examples}\label{s:four}

Various classes of stationary sequences are covered by our main theorem, such as squared GARCH processes, linear processes, moving maxima and ARMAX processes (see Basrak et al.~\cite{BKS} and Krizmani\'{c}~\cite{Kr14}). Here we present in detail only linear processes and moving maxima processes, and for the latter we show that Theorem~\ref{t:functconvergence} fails to hold under the standard $M_{1}$ topology on $D[0,1], \mathbb{R}^{2})$.

\begin{ex}(Linear processes)
Let $(Z_{i})_{i \in \mathbb{Z}}$ be an i.i.d.~sequence of regularly varying random variables with index $\alpha \in (0,2)$. Consider the linear process of the form
\begin{equation}\label{e:madefn}
X_{i} = \sum_{j=0}^{\infty}\varphi_{j}Z_{i-j}, \qquad i \in \mathbb{Z},
\end{equation}
where the sequence of real numbers $(\varphi_{j})$ is such that
$$ \sum_{j=0}^{\infty}|\varphi_{j}|^{\delta} < \infty \quad \textrm{for some} \ 0 < \delta < \alpha,\,\delta \leq 1$$
(under this condition the series in (\ref{e:madefn}) is a.s.~convergent, see Resnick~\cite{Re87}, Section 4.5). We assume also $\sum_{j=0}^{\infty}\varphi_{j} \neq 0$. It holds that
$$ \lim_{x \to \infty} \frac{\PP (|X_{0}|>x)}{\PP(|Z_{0}|>x)} = \sum_{j=0}^{\infty}|\varphi_{j}|^{\alpha}$$
(Cline~\cite{Cl83}, Theorem 2.3).
The tail process $(Y_{n})$ of the linear process $(X_{n})$ is of the following form:
$$ Y_{n} = \frac{\varphi_{n+K}}{|\varphi_{K}|} |Y_{0}| \Theta^{Z}, \qquad n \in \mathbb{Z},$$
where $\Theta^{Z}$ is an $\{-1,1\}$--valued random variable such that $\PP (\Theta^{Z} =1)=p$ and $\PP (\Theta^{Z}=0)=q$, and $K$ is an integer valued random variable, independent of $\Theta^{Z} = Y_{0}/|Y_{0}|$, with distribution given by
$$ \PP (K=j) = \frac{|\varphi_{j}|^{\alpha}}{\sum_{i=0}^{\infty}|\varphi_{i}|^{\alpha}}, \qquad j \geq 0$$
(see Meinguet and Segers~\cite{MeSe10}, Example 9.2). In~\cite{MeSe10} the extremal index $\theta$ from (\ref{E:theta:spectral}) was also computed, and it is given by
$$ \theta = \frac{\max_{j \geq 0}|\varphi_{j}|^{\alpha}}{\sum_{j=0}^{\infty}|\varphi_{j}|^{\alpha}}.$$
 Then, as noted in Basrak et al.~\cite{BaPlSo}, Section 3.3, it holds that
 $$ \sum_{j} \delta_{\eta_{j}} \eqd \sum_{j} \delta_{\Theta^{Z} \varphi_{j}/ \max_{i \geq 0}|\varphi_{i}| }.$$
 Assume now all coefficients $\varphi_{j}$ are of the same sign. This assumption ensures that the tail process $(Y_{n})$ has no two values of the opposite sign. Theorem~\ref{t:functconvergence} directly applies to the case of a finite order linear processes, since in this case all conditions in the mentioned theorem are satisfied (see Basrak et al.~\cite{BKS}, Example 4.3). Hence for $X_{i} = \sum_{j=0}^{m}\varphi_{j}Z_{i-j}$, according to Remark~\ref{r:jointdepend}, we have
 $$ V(\,\cdot\,) = \left\{ \begin{array}{ll}
        \displaystyle  \frac{\sum_{j=0}^{m}\varphi_{j}}{\kappa} \sum_{T_{i} \leq\,\cdot} P_{i} \Theta_{i}^{Z}      , & \alpha \in (0,1),\\[1em]
           \displaystyle \lim_{u \to 0} \bigg( \frac{1}{\kappa} \sum_{T_{i} \leq\,\cdot} P_{i} \Theta_{i}^{Z} \sum_{j=0}^{m} \varphi_{j} 1_{\{ P_{i}|\varphi_{j}| / \kappa \}} - (\,\cdot\,) \int_{u < |x| \leq 1}x \mu(\rmd x) \bigg)   , & \alpha \in [1,2),
\end{array}\right.$$
and
$$
W(\,\cdot\,) = \frac{1}{\kappa} \bigvee_{T_{i} \leq\,\cdot} P_{i} \bigvee_{j=0}^{m} \Theta_{i}^{Z} \varphi_{j} \vee 0,
$$
where $ \kappa := \max_{0 \leq j \leq m}|\varphi_{j}|$ and $\Theta_{i}^{Z}$, $i \geq 1$, are i.i.d.~copies of $\Theta^{Z}$. The characteristic triple of the process $V$ is $(0, \nu', \gamma)$, where
$$ \nu'(\rmd x) = \frac{( \sum_{j=0}^{m}|\varphi_{j}|)^{\alpha}}{\sum_{j=0}^{m}|\varphi_{j}|^{\alpha}} A(x) \alpha |x|^{-\alpha-1} \rmd x, $$
$$ A(x) = \left\{ \begin{array}{ll}
     p 1_{(0,\infty)}(x) + q 1_{(-\infty,0)}(x), & \textrm{if all coefficients} \ \varphi_{j} \ \textrm{are non-negative},\\[0.6em]
      q1_{(0,\infty)}(x) + p 1_{(-\infty,0)}(x), & \textrm{if all coefficients} \ \varphi_{j} \ \textrm{are non-positive},
\end{array}\right.$$
with $p$ and $q$ as in (\ref{e:mu}), and $\gamma$ as in (\ref{e:gamma}) with
$$c_{+}= \left\{ \begin{array}{ll}
         \frac{p}{\kappa^{\alpha}} (\sum_{j=0}^{m} |\varphi_{j}| )^{\alpha}, & \textrm{if all coefficients} \ \varphi_{j} \ \textrm{are non-negative},\\[0.6em]
         \frac{q}{\kappa^{\alpha}} (\sum_{j=0}^{m} |\varphi_{j}| )^{\alpha}, & \textrm{if all coefficients} \ \varphi_{j} \ \textrm{are non-positive},
\end{array}\right.$$
$$ c_{-} = \left\{ \begin{array}{ll}
         \frac{q}{\kappa^{\alpha}} (\sum_{j=0}^{m} |\varphi_{j}| )^{\alpha}, & \textrm{if all coefficients} \ \varphi_{j} \ \textrm{are non-negative},\\[0.6em]
         \frac{p}{\kappa^{\alpha}} (\sum_{j=0}^{m} |\varphi_{j}| )^{\alpha}, & \textrm{if all coefficients} \ \varphi_{j} \ \textrm{are non-positive}.
\end{array}\right.$$
For $\alpha=1$, $\gamma$ reduces to
$$ \gamma = \frac{q-p}{\sum_{j=0}^{m}|\varphi_{j}|} \sum_{j=0}^{m}\varphi_{j} \log \bigg( \bigg| \frac{1}{\varphi_{j}} \sum_{i=0}^{m}\varphi_{i} \bigg| \bigg).$$
The exponent measure of the process $W$ is of the form
$$\nu''(\rmd x) =  \frac{r \kappa^{\alpha}}{\sum_{j=0}^{m}|\varphi_{j}|^{\alpha}}
 \alpha x^{-\alpha -1} 1_{(0,\infty)}(x)\,\rmd x,$$
 where
$$ r= \left\{ \begin{array}{ll}
         p, & \textrm{if all coefficients} \ \varphi_{j} \ \textrm{are non-negative},\\[0.6em]
        q, & \textrm{if all coefficients} \ \varphi_{j} \ \textrm{are non-positive}.
\end{array}\right.$$
One can show that these expressions are consistent with the results obtained in Krizmanic~\cite{Kr18} for linear processes from a regularly varying distribution with index $\alpha \in (0,2)$ (with some differences due to different centering and normalizing sequences that are being used).

For infinite order linear processes with all coefficients of the same sign the idea is to approximate them by a sequence of finite order linear processes for which Theorem~\ref{t:functconvergence} applies, and to show that the error of approximation is negligible in the limit (for an example of this procedure see Krizmani\'{c}~\cite{Kr18}, Section 4).
\end{ex}

\begin{ex}\label{ex:WM1M1} (Moving maxima).
Let $(Z_n)_{n \in \mathbb{Z}}$ be a sequence of i.i.d.~Fr\'{e}chet random variables with shape parameter $\alpha \in (0,2)$, i.e. $\PP(Z_{n} \leq x) = e^{-x^{-\alpha}}$ for $x>0$. Hence $Z_{n}$ is regularly varying with index $\alpha$. For simplicity we consider only the case $\alpha \in (0,1)$.
Consider the finite order moving maxima
$$
X_n = \max_{i=0,\ldots,m}\{c_{i}Z_{n-i}\}, \quad {n \in \mathbb{Z}},
$$
where $m \in \mathbb{N}$ and $c_{0}, \ldots, c_{m}$ are nonnegative constants such that at least $c_{0}$ and $c_{m}$ are not equal to zero. Take a sequence of positive real numbers $(a_{n})$ such that
$n \PP (X_{0}>a_{n}) \to 1$ as $n \to \infty$. Then
\begin{equation}\label{e:clinemm}
\lim_{n \to \infty} n \PP(Z_{0}>a_{n}) = \frac{1}{\sum_{i=0}^{m}c_{i}^{\alpha}}
\end{equation}
(Cline~\cite{Cl83}, Theorem 2.3).
The random process $(X_{n})$ is jointly regularly varying with index $\alpha$ (Tafro~\cite{TafroPHD}, Example 2.1.12). Since the sequence $(X_{n})$ is $m$--dependent, it follows immediately that Conditions~\ref{c:mixcond1} and~\ref{c:mixcond2} hold (see for instance Example 5.1 in Krizmani\'{c}~\cite{Kr14}).
Therefore $(X_{n})$ satisfies all conditions of Theorem~\ref{t:functconvergence}, and the corresponding stochastic process $L_{n}$ converges in distribution in $D([0,1], \mathbb{R}^{2})$ under the weak $M_{1}$ topology.

Next we show that $L_{n}$ does not converge in distribution under the standard $M_{1}$ topology on $D([0,1], \mathbb{R}^{2})$. This shows that the weak $M_{1}$ topology in Theorem~\ref{t:functconvergence} in general can not be replaced by the standard $M_{1}$ topology. In showing this we use, with appropriate modifications, a combination of arguments used by Basrak and Krizmani\'{c}~\cite{BaKr} in their Example 4.1 and Avram and Taqqu~\cite{AvTa92} in their Theorem 1 (see also Example 5.1 in Krizmani\'{c}~\cite{Kr17}).

For simplicity take $m=1$ and $c_{0}=c_{1}=1$.
We have $X_{n}= Z_{n} \vee Z_{n-1}$ and $L_{n}(t) = (V_{n}(t), W_{n}(t))$, where
$$ V_{n}(t) = \sum_{j=1}^{\floor{nt}} \frac{X_{j}}{a_{n}} \quad \textrm{and} \quad W_{n}(t) = \bigvee_{j=1}^{\floor{nt}} \frac{X_{j}}{a_{n}}.$$
Let
 $$ G_{n}(t) := V_{n}(t) - 2W_{n}(t), \quad t \in [0,1].$$
The first step is to show that $G_{n}$ does not converge in distribution in $D([0,1], \mathbb{R})$ endowed with the (standard) $M_{1}$ topology. For this, according to Skorohod~\cite{Sk56} (see also Proposition 2 in Avram and Taqqu~\cite{AvTa92}), it suffices to show that
 \begin{equation}\label{e:osc1}
 \lim_{\delta \to 0} \limsup_{n \to \infty} \PP ( \omega_{\delta}(G_{n}) > \epsilon ) > 0
 \end{equation}
 for some $\epsilon >0$, where
 $$ \omega_{\delta}(x) = \sup_{{\footnotesize \begin{array}{c}
                                t_{1} \leq t \leq t_{2} \\
                                0 \leq t_{2}-t_{1} \leq \delta
                              \end{array}}
} M(x(t_{1}), x(t), x(t_{2}))$$
($x \in D([0,1], \mathbb{R}), \delta >0)$ and
$$ M(x_{1},x_{2},x_{3}) = \left\{ \begin{array}{ll}
                                   0, & \ \ \textrm{if} \ x_{2} \in [x_{1}, x_{3}], \\
                                   \min\{ |x_{2}-x_{1}|, |x_{3}-x_{2}| \}, & \ \ \textrm{otherwise},
                                 \end{array}\right.$$
Note that $M(x_{1},x_{2},x_{3})$ is the distance from $x_{2}$ to $[x_{1}, x_{3}]$, and $\omega_{\delta}(x)$ is the $M_{1}$ oscillation of $x$.

Let $i'=i'(n)$ be the index at which $\max_{1 \leq i \leq n-1}Z_{i}$ is obtained. Fix $\epsilon >0$ and introduce the events
 $$A_{n,\epsilon} = \{ Z_{i'} > \epsilon a_{n} \} = \Big\{ \max_{1 \leq i \leq n-1}Z_{i} > \epsilon a_{n}\Big\}$$
 and
 $$ B_{n,\epsilon} = \{Z_{i'}>\epsilon a_{n} \ \textrm{and} \ \exists\,l
 \neq 0, -i' \leq l \leq 1, \ \textrm{such that} \ Z_{i'+l} > \epsilon a_{n} / 4 \}.$$
Using the facts that $(Z_{i})$ is an i.i.d.~sequence and $n
 \PP(Z_{1}> c a_{n}) \to c^{-\alpha}/2$ as $n \to \infty$ for $c>0$
 (which follows from the regular variation property of $Z_{1}$ and (\ref{e:clinemm})) we get
 \begin{equation}\label{e:limAn}
  \lim_{n \to \infty}\PP(A_{n,\epsilon}) = 1 - e^{-\epsilon^{-\alpha}/2} \end{equation}
and
 \begin{equation}\label{e:limBn}
 \limsup_{n \to \infty} \PP(B_{n,\epsilon})  \leq  \frac{\epsilon^{-2\alpha}}{4^{1-\alpha}}
 \end{equation}
(see Example 5.1 in Krizmani\'{c}~\cite{Kr14}).
On the event $A_{n,\epsilon} \setminus B_{n,\epsilon}$
one has $Z_{i'} > \epsilon a_{n}$ and $Z_{i'+l} \leq \epsilon a_{n}/4$ for every $l \neq 0$, $-i' \leq l \leq 1$, so that
$$ W_{n} \Big( \frac{i'}{n} \Big) = W_{n} \Big( \frac{i'+1}{n} \Big) = \frac{Z_{i'}}{a_{n}} > \epsilon \quad \textrm{and} \quad W_{n} \Big( \frac{i'-1}{n} \Big) = \bigvee_{j=0}^{i'-1} \frac{Z_{j}}{a_{n}} \leq \frac{\epsilon}{4}.$$
Therefore after standard calculations we obtain
\begin{equation}\label{e:inc1}
  \Big| G_{n} \Big( \frac{i'}{n} \Big) - G_{n} \Big( \frac{i'-1}{n} \Big) \Big| = \Big| - \frac{Z_{i'}}{a_{n}} + 2 W_{n} \Big(\frac{i'-1}{a_{n}} \Big) \Big| > \frac{\epsilon}{2}
\end{equation}
and
\begin{equation}\label{e:inc2}
  \Big| G_{n} \Big( \frac{i'+1}{n} \Big) - G_{n} \Big( \frac{i'}{n} \Big) \Big| = \frac{Z_{i'}}{a_{n}} > \epsilon.
\end{equation}
On the set $A_{n,\epsilon} \setminus B_{n,\epsilon}$ it also holds that
$$ G_{n} \Big( \frac{i'}{n} \Big) \notin \Big[ G_{n} \Big( \frac{i'-1}{n} \Big), G_{n} \Big( \frac{i'+1}{n} \Big) \Big],$$
which implies that
\begin{eqnarray*}
  M \Big( G_{n} \Big( \frac{i'-1}{n} \Big), G_{n} \Big( \frac{i'}{n} \Big), G_{n} \Big( \frac{i'+1}{n} \Big) \Big) & &  \\[0.8em]
   & \hspace*{-20em} =& \hspace*{-10em} \min \bigg\{  \Big| G_{n} \Big( \frac{i'}{n} \Big) - G_{n} \Big( \frac{i'-1}{n} \Big) \Big|, \Big| G_{n} \Big( \frac{i'+1}{n} \Big) - G_{n} \Big( \frac{i'}{n} \Big) \Big| \bigg\}.
\end{eqnarray*}
Taking into account (\ref{e:inc1}) and (\ref{e:inc2}) we obtain
\begin{eqnarray*}
  \omega_{2/n}(G_{n}) & = & \sup_{{\footnotesize \begin{array}{c}
                                t_{1} \leq t \leq t_{2} \\
                                0 \leq t_{2}-t_{1} \leq 2/n
                              \end{array}}
} M(G_{n}(t_{1}), G_{n}(t), G_{n}(t_{2})) \\[0.8em]
   & \geq & M \Big( G_{n} \Big( \frac{i'-1}{n} \Big), G_{n} \Big( \frac{i'}{n} \Big), G_{n} \Big( \frac{i'+1}{n} \Big) \Big) > \frac{\epsilon}{2}
\end{eqnarray*}
on the event $A_{n,\epsilon} \setminus B_{n,\epsilon}$. Therefore, since $\omega_{\delta}$ is nondecreasing in $\delta$, it holds that
 \begin{eqnarray}\label{e:oscM1}
  \nonumber \liminf_{n \to \infty} \PP(A_{n,\epsilon} \setminus B_{n,\epsilon}) & \leq & \liminf_{n \to \infty}
 \PP (\omega_{2/n} (G_{n}) >  \epsilon /2)\\[0.4em]
 & \leq &   \lim_{\delta \to 0} \limsup_{n \to \infty}  \PP (\omega_{\delta} (G_{n}) >  \epsilon/2).
 \end{eqnarray}
Note that $x^{2\alpha}(1-e^{-x^{-\alpha}/2})$ tends to infinity as $x \to \infty$, and therefore we can find $\epsilon >0$ such that  $\epsilon^{2\alpha}(1-e^{-\epsilon^{-\alpha}/2}) > 4^{\alpha-1}$, i.e.
 $$ 1-e^{-\epsilon^{-\alpha}/2} > \frac{4^{\alpha-1}}{ \epsilon^{2\alpha}}.$$
 For this $\epsilon$, by relations (\ref{e:limAn}) and (\ref{e:limBn}), it holds that
 $$\lim_{n \to \infty} \PP(A_{n,\epsilon}) > \limsup_{n \to \infty} \PP(B_{n,\epsilon}),$$
 i.e.
 $$  \liminf_{n \to \infty} \PP(A_{n,\epsilon} \setminus B_{n,\epsilon}) \geq \lim_{n \to \infty}\PP(A_{n,\epsilon}) - \limsup_{n \to \infty} \PP(B_{n,\epsilon}) >0.$$
Thus by (\ref{e:oscM1}) we obtain
$$ \lim_{\delta \to 0} \limsup_{n \to \infty}  \PP (\omega_{\delta} (G_{n}) >  \epsilon/2) > 0$$
and (\ref{e:osc1}) holds, i.e. $G_{n}$ does not converge in distribution in $D([0,1], \mathbb{R})$ endowed with the (standard) $M_{1}$ topology.

If $L_{n}$ would converge in distribution to some $L = (V, W)$ in the standard $M_{1}$ topology on $D([0,1], \mathbb{R}^{2})$, then using the fact that linear combinations of the coordinates are continuous in the same topology (see Theorem 12.7.1 and Theorem 12.7.2 in Whitt~\cite{Whitt02}) and the continuous mapping theorem, we would obtain that $G_{n} = V_{n} - 2W_{n}$ converges to $V - 2W$ in $D([0,1], \mathbb{R})$ endowed with the standard $M_{1}$ topology, which is impossible, as is shown above.

\end{ex}

\section*{Acknowledgements}
This work has been supported in part by Croatian Science Foundation under the project 3526 and by University of Rijeka under the project numbers 13.14.1.2.02 and 17.15.2.2.01.

\end{document}